\documentclass[11pt,reqno]{amsart}
\usepackage[foot]{amsaddr}
\usepackage{amsmath} %[intlimits]
\usepackage{amsfonts,amssymb,amsthm,bm}
\usepackage{amsrefs}
\usepackage{esint}

\usepackage{enumitem}
\usepackage[utf8]{inputenc}
\usepackage[T1]{fontenc}
\usepackage{stackrel}

\usepackage{fullpage}

\usepackage[dvipsnames]{xcolor}

\usepackage{hyperref}

\newcommand{\op}[1]{\operatorname{#1}}

\newcommand\abs[1]{\left|#1\right|} % absolute value
\DeclareMathOperator{\diam}{diam} % diameter of a set
\DeclareMathOperator{\dist}{\mathit{d}} % distance
\let\emptyset\varnothing % nicer empty set
\DeclareMathOperator{\Lip}{Lip} % Lip
\newcommand{\aabs}[1]{|#1|}

\newcommand{\Norm}[2]{\|#1\|_{#2}}

\newcommand{\calL}{\mathcal{L}}

\newcommand{\calQ}{\mathcal{Q}}

\newcommand{\calW}{\mathcal{W}}

\newcommand{\N}{\mathbb{N}}

\newcommand{\R}{\mathbb{R}}

\newcommand{\BMO}{\mathrm{BMO}}

\newcommand{\VC}{\dot{\operatorname{VC}}}

\def\XXint#1#2#3{\mkern3mu{\setbox0=\hbox{$#1{#2#3}{\int}$ }
\vcenter{\hbox{$#2#3$ }}\kern-.6\wd0}}

\theoremstyle{plain}
\newtheorem{theorem}{Theorem}
\newtheorem{propo}[theorem]{Proposition}

\newtheorem{lemma}[theorem]{Lemma}
\newtheorem{corollary}[theorem]{Corollary}
\newtheorem{problem}[theorem]{Problem}

\numberwithin{theorem}{section}

\newtheorem{remark}[theorem]{Remark}

\theoremstyle{definition}
\newtheorem{dfn}[theorem]{Definition}
 \title{Traces of vanishing H\"older spaces}
 
 \author{Kaushik Mohanta}
 
 \address[K.M.]{University of Jyv\"askyl\"a, Department of Mathematics and Statistics, P.O. Box 35, FI-40014 University of Jyv\"askyl\"a, Finland} 
 \email{kaushik.k.mohanta@jyu.fi}

 \author[Carlos Mudarra]{Carlos Mudarra}
\address[C.M.]{Department of Mathematical Sciences, Norwegian University of Science and Technology, 7941 Trondheim, Norway}
\email{carlos.mudarra@ntnu.no}

\author[Tuomas Oikari]{Tuomas Oikari}
\address[T.O.]{University of Helsinki, Department of Mathematics and Statistics, P.O.B. 68, FI-00014, Helsinki, Finland}
\email{tuomas.v.oikari@helsinki.fi}

\date{\today}

\keywords{Extension operator, trace space, H\"older space, vanishing H\"older space}

\subjclass[2020]{26B35, 42B35, 46E35, 26B05, 46T20}
\begin{document}

\maketitle

\begin{abstract} 
	For an arbitrary subset $E\subset\R^n,$ we introduce and study the three vanishing subspaces of the H\"older space $\dot{C}^{0,\omega}(E)$ consisting of those functions for which the ratio $|f(x)-f(y)|/\omega(|x-y|)$ vanishes, when $(1)$ $|x-y|\to 0$ , $(2)$ $|x-y|\to\infty$ or $(3)$ $\min(|x|,|y|)\to\infty.$ 
	We prove that the Whitney extension operator maps each of these vanishing subspaces from $E$ to the corresponding vanishing spaces defined on the whole ambient space $\R^n.$
	In fact, this follows as the zeroth order special case of a more general problem involving higher order derivatives.
	As a consequence, we obtain complete characterizations of approximability of H\"older functions $\dot{C}^{0,\omega}(E)$ by Lipschitz and boundedly supported functions.
\end{abstract}

\section{Introduction and main results}

\subsection{Introduction}\label{sect:intro1}

The vanishing H\"{o}lder spaces $\VC_{\Gamma}^{\omega}(X,Y),$ for the scales $\Gamma\in\{\op{small}, \op{large}, \op{far}\},$ for normed spaces $X,Y$  and for many moduli of continuity $\omega,$ were recently defined and studied by the two last named authors \cite{MudOik24}. In particular, these vanishing H\"older spaces were shown to often provide a complete description of those H\"older functions approximable by Lipschitz or smooth or boundedly supported functions. In this article we continue to study the vanishing H\"{o}lder spaces $\VC_{\Gamma}^{\omega}(X,Y),$ when $X=\R^n$ but $Y$ is allowed to be an arbitrary normed space. 

\smallskip

We are interested in understanding when a given vanishing scale restricted to a proper subset $E\subset\R^n,$ admits a bounded linear extension operator $L:\VC_{\Gamma}^{\omega}(E,Y) \to \VC_{\Gamma}^{\omega}(\R^n,Y)$ to the whole ambient space. Our answer is that for a completely arbitrary subset $E\subset\R^n$,  a single linear bounded extension operator exists and works simultaneously for all the scales $\Gamma\in\{\op{small}, \op{large}, \op{far}\}.$ We use the classical Whitney extension operator.

We do not only consider the possibility of extending the function itself, but also its \textit{putative} derivatives. In particular, we provide a full description of exactly when a jet $\mathcal{A}\in \dot{\op{J}}^{m,\omega}(E,Y)$ is obtained by restricting a function $F \in \dot{C}^{m,\omega}_{\Gamma}(\R^n,Y).$ This amounts to showing that the Whitney extension operator maps the vanishing jet spaces $\dot{\op{J}}^{m,\omega}_{\Gamma}(E,Y)$ to $ \dot{C}^{m,\omega}_{\Gamma}(\R^n,Y),$  i.e. preserves separately each of the vanishing scales. See Section \ref{sect:intro2} for the precise definitions and statements of these results.

\smallskip

One perspective to our results is as follows.
The scale of H\"older spaces $\dot C^{0,\omega}$ can be seen to measure the smoothness of a function relative to the modulus $\omega(t)>0,$ and when $\omega(t) = 0$ the space $\BMO$ of bounded mean oscillations is often substituted as the zeroth order way of measuring smoothness.  As opposed to real-valued H\"older functions $\dot C^{0,\omega}(E,\R^n),$ with $E\subset \R^n$ arbitrary,  being extendable to the whole ambient space $\dot C^{0,\omega}(\R^n,\R)$ by the classical infimal convolution formula $x\mapsto \inf_{y\in E}\{ f(y) + M \omega(|x-y|)\},$ for example, this is far from being true for functions in $\BMO.$ Indeed, given a connected open set (a domain) $\Omega\subset\R^n,$ a classical theorem of Jones \cite{Jon80Ext} tells us that a bounded linear extension operator $L:\BMO(\Omega)\to\BMO(\R^n)$ exists if and only if $\Omega$ is uniform. With $\BMO,$ geometry of the domain appears naturally.  A recent result of Butaev and Gafni \cite[Theorem 3]{ButDaf21} (see also \cite{ButDaf23}) extends Jones' result by providing a single bounded extension operator on $\BMO(\Omega)$ and also on several of its subspaces determined by approximability by nice functions, or by having appropriate vanishing mean oscillations. Our results are in the same spirit as Butaev's and Gafni's, but opposed to Jones';  we consider vanishing subspaces, but get rid of all geometry by reintroducing an order of smoothness $\omega(t)>0.$

\smallskip

It follows from our results that a single linear and bounded extension operator exists for the intersection $\VC^{\omega}$ of the three vanishing scales. 
In particular, for the H\"older moduli $\alpha(t) := t^{\alpha},$ for $\alpha\in(0,1),$ \cite[Theorem 1.13]{MudOik24} states that $\VC^{\alpha}(\R^n)=\op{VMO}^{\alpha}(\R^n)\footnote{We remark that notationally $\op{VMO}^{\alpha} =\op{CMO}^{\alpha}$ for Guo et al.} ,$ where the latter fractional vanishing mean oscillation space was defined in Guo et al. \cite[Theorem 1.7]{GHWY21} and shown to characterizes the off-diagonal $L^p\to L^q$ (for $1<p<q<\infty$) compactness of commutators of many singular integral operators. We also point out that the identification $\dot C^{0,\alpha}(\R^n,\R) = \op{BMO}^{\alpha}(\R^n,\R)$  of the H\"older spaces with the spaces of fractional bounded mean oscillation is due to N. G. Meyers \cite{Mey1964}. 
Moreover, the classes $\VC^{0,\alpha}_{\op{small}}$ are also called the \textit{little H\"older spaces}. These are fundamental in the study of Lipschitz algebras in metric spaces; we refer to the monographs by Weaver \cite[Chapters 4 and 8]{Weaverbook}, and for an exposition of these little H\"older spaces in the setting of Ahlfors regular sets we refer to D. Mitrea, I. Mitrea, and M. Mitrea \cite[Chapter 3]{Mitreabook}. 
\smallskip

For the extension of $\dot C^{0,\omega}$ mappings between subsets of Hilbert spaces, preserving the $\dot C^{0,\omega}$-seminorms, see Gr\"unbaum and Zarantonello \cite{GruZar68}, and the monograph \cite{BenyaminiLindenstrauss} by Benyamini and Lindenstrauss. The fundamental starting point for $C^{m}$ extension of jets in $\R^n$ is the celebrated Whitney Extension Theorem \cite{Wh34}. For Whitney-type $C^{m,\omega}$ extensions of jets in $\R^n,$ we refer to the work of Glaeser \cite{Gl58}, and for Whitney-type extensions of jets generated by Sobolev functions in $\R^n$, see, for instance, the recent paper of Shvartsman \cite{Shv17}. For infinite dimensional results on $C^{1,\omega}$ extension of jets, see the recent paper \cite{AzaMud21} by Azagra and the second named author of the present paper. For extension results for  functions (instead of jets) of order $C^m$ or $C^{m,\omega}$, or for Sobolev functions, see for instance the papers by Brudnyi and Shvartsman \cite{BruShv01}, by Fefferman \cite{Fef05,Fef06}, by Fefferman, Israel and Luli \cite{FefIsrLul14}.

\subsection*{Acknowledgements}
We are grateful to the anonymous referee for carefully reading the manuscript and for their comments, which led to improvements in both clarity and presentation.

K.M. was supported by the Academy of Finland (project no. 323960) and by the Academy of Finland via Centre of Excellence in Analysis and Dynamics Research (project no. 346310)
C.M. was supported by grant no. 334466 of the Research Council of Norway, ``Fourier Methods and Multiplicative Analysis''. 
T.O. was supported by the Finnish Academy of Science and Letters, and by the MICINN (Spain) grant no. PID2020-114167GB-I00.

\subsection{Basic definitions and main results}\label{sect:intro2}

\begin{dfn} Let $\omega:(0,\infty)\to (0,\infty)$ be a modulus of continuity, $E\subset \R^n$ an arbitrary set, $m \in \N \cup \lbrace 0 \rbrace$, and $V$ a normed space. By an $m$-jet on $E$ (to $V$) we simply understand a family of $m+1$ functions $\lbrace A_k: E \to \mathcal{L}^k(\R^n, V) \rbrace_{k=0}^m.$ For us $\mathcal{L}^k(\R^n, V)$ denotes the vector space of all symmetric $k$-linear mappings from $\R^n$ to $V$, and the norm we are using on $\mathcal{L}^k(\R^n, V)$ is the one given by 
	$$
	\|T\|:= \sup\lbrace \|T(u_1,\ldots,u_k)\|_V \, : \, u_1, \ldots, u_k \in \mathbb{S}^{n-1} \rbrace.
	$$

Then we define \textit{the trace jet space of $\dot{C}^{m,\omega}(\R^n,V)$ to $E$}, denoted by $\dot{\op{J}}^{m,\omega}(E,V)$, as the vector space of all $m$-jets $\lbrace A_k \rbrace_{k=0}^m$ on $E$ so that
\begin{equation}\label{eq:norm:jet}
\begin{split}
		&\|\lbrace A_k \rbrace_{k=0}^m \|_{\dot{\op{J}}^{m,\omega}(E,V)} \\ 
	&: = \sup \left\lbrace \frac{\|A_k(x)-\sum_{j=0}^{m-k} \frac{1}{j!}A_{k+j}(y)(x-y)^j\|}{\omega(|x-y|)|x-y|^{m-k}} \: : \: x,y\in E, \, k=0, \ldots, m \right\rbrace < \infty.
\end{split}
\end{equation}
\end{dfn}

\smallskip
It is convenient to clarify that each term $A_{k+j}(y)(x-y)^j$ in the sum of \eqref{eq:norm:jet} is understood as the symmetric $k$-linear mapping $(\R^n)^k \to V$ given by:
$$
A_{k+j}(y)(x-y)^j(u_1,\ldots,u_k):= A_{k+j}(y) \Big ( \overbrace{x-y, \ldots, x-y}^{j \text{ times} }  ,  u_1, \ldots, u_k \Big ), \quad (u_1,\ldots,u_k) \in (\R^n)^k.
$$
In particular, the comparison of the sum of these mappings with $A_k(x)$ in \eqref{eq:norm:jet} is well-defined. In addition, since $A_{k+j}(y) $ is a symmetric $(k+j)$-linear mapping $(\R^n)^{k+j} \to V$, one can distribute the $j$-many vectors $x-y$ among the $k+j$ entries in any order, resulting in the same symmetric $k$-linear mapping as above.

Moreover, it is clear that $\|  \cdot  \|_{\dot{\op{J}}^{m,\omega}(E,V)}$ defines a seminorm on $\dot{\op{J}}^{m,\omega}(E,V)$, and it can be made an actual norm if we fix a point $x_0 \in E$ and define
$$
||| \lbrace A_k \rbrace_{k=0}^m |||_{\op{J}^{m,\omega}(E,V)} := \|\lbrace A_k \rbrace_{k=0}^m \|_{\dot{\op{J}}^{m,\omega}(E,V)} + \max_{k=0,\ldots,m} \| A_k(x_0)\|.
$$
Moreover, if $V$ is a Banach space, this norm is complete, and $ \big( \op{J}^{m,\omega}(E,V),  ||| \cdot|||_{\op{J}^{m,\omega}(E,V)} \big)$ becomes a Banach space.

\begin{remark}\label{rem:m=0}
{\em In the case $m=0,$ the trace space $\dot{\op{J}}^{0,\omega}(E,V)$ coincides with the homogeneous H\"older space $\dot{C}^{0,\omega}(E,V)$, consisting of those functions $f:E \to V$ such that
	\[
	\Norm{f}{\dot{C}^{0,\omega}(E,V)} = \sup_{ \substack{x \neq y \\ x,y \in E}}\frac{\|f(x)-f(y)\|}{\omega(\aabs{x-y})}<\infty.
	\]
	}
\end{remark}

Let us now look at functions that are defined everywhere in $\R^n.$ 
\begin{dfn} Let $\omega:(0,\infty)\to (0,\infty)$ be a modulus of continuity, $m \in \N \cup \lbrace 0 \rbrace,$ and $V$ a normed space. Then $\dot{C}^{m,\omega}(\R^n,V)$ consists of those $F:\R^n \to V$ of class $C^m(\R^n,V)$ so that
	\[
	 \Norm{F}{\dot{C}^{m,\omega}(\R^n,V)} = \sup_{ \substack{x \neq y \\ x,y \in \R^n}}\frac{\|D^m F(x)-D^m F (y)\|}{\omega(|x-y|)}<\infty. 
	\]
\end{dfn}
Here $D^m F: \R^n \to \mathcal{L}^m(\R^n,V) $ denotes the $m$th (total) derivative of $F.$ Under basic assumptions on the modulus $\omega,$ a version of Whitney's extension theorem says that every $m$-jet $\lbrace A_k \rbrace_{k=0}^m \in \dot{\op{J}}^{m,\omega}(E,V)$ is \textit{the restriction} of some $F\in \dot{C}^{m,\omega}(\R^n,V)$ to $E,$ in the sense that
$$
D^k F (y)= A_k(y), \quad y\in E, \: k=0,\ldots,m.
$$
We recall the construction of the Whitney extension operator and give appropriate references in Section \ref{section:sufficienty}.

In this article, we are insterested in subclasses of the space $\dot{C}^{m,\omega}(\R^n,V)$ given by the following vanishing conditions.   
\begin{dfn}\label{def:scales} Let $\omega:(0,\infty)\to (0,\infty)$ be a modulus of continuity, $m\in \N \cup \lbrace 0 \rbrace$, $n\in \N,$ and $V$ a normed space. We define the vanishing scales
	\begin{align*}
		\VC^{m,\omega}_{\op{small}}(\R^n,V) := \Big\{  F\in  \dot{C}^{m,\omega}(\R^n,V) \, : \, \lim_{\delta \to 0}  \sup_{ \substack{x\not=y\in \R^n \\ |x-y| \leq \delta}}\frac{\|D^m F(x)-D^m F(y)\|}{\omega(|x-y|)}  = 0 \Big \},
	\end{align*}
		\begin{align*}
		\VC^{m,\omega}_{\op{large}}(\R^n,V) := \Big\{  F\in  \dot{C}^{m,\omega}(\R^n,V) \, : \, \lim_{\delta \to \infty}  \sup_{ \substack{x, y\in \R^n \\ |x-y| \geq \delta}}\frac{\|D^m F(x)-D^m F(y)\|}{\omega(|x-y|)}  = 0 \Big \},
	\end{align*}
\begin{align*}
	\VC^{m,\omega}_{\op{far}}(\R^n,V) := \Big\{ F\in  \dot{C}^{m,\omega}(\R^n,V) \, : \, \lim_{\delta \to \infty} \sup_{\substack{x\not=y\in \R^n \\ \min(|x|,|y|) \geq \delta}} \frac{\|D^m F(x)- D^m F(y)\|}{\omega(|x-y|)}  = 0 \Big\}.
\end{align*}
Finally we take the intersection of all the scales,
	\[
	\VC^{m,\omega}(\R^n,V) := \VC^{m,\omega}_{\op{small}}(\R^n,V) \cap \VC^{m,\omega}_{\op{far}}(\R^n,V) \cap \VC^{m,\omega}_{\op{large}}(\R^n,V).
	\]
\end{dfn}

Let us next consider the following questions.
\begin{problem}\label{problem:ext}
Given an $m$-jet $\lbrace A_k: E \to \mathcal{L}^k(\R^n, V) \rbrace_{k=0}^m$ on $E$,  a modulus of continuity $\omega$, and one of the scales $\Gamma \in \lbrace \op{small}, \op{large}, \op{far} \rbrace$ defined above, what necessary and sufficient conditions guarantee the existence of $F\in \VC^{m,\omega}_{\Gamma}(\R^n,V)$ so that $D^k F$ restricted to $E$ agrees with $A_k,$ for each $k=0,\ldots,m$? Does there exist an extension operator that works simultaneously for all of the three vanishing subspaces? Can such an extension be defined by means of a  linear operator?
\end{problem}

For an $m$-jet $ \lbrace A_k \rbrace_{k=0}^m \in \dot{\op{J}}^{m,\omega}(E,V)$, let us denote
\begin{equation}\label{eq:vanishing:remainders}
R  \left( \lbrace A_k \rbrace_{k=0}^m, x,y \right):= \max_{0 \leq k \leq m} \frac{\|A_k(x)-\sum_{j=0}^{m-k} \frac{1}{j!}A_{k+j}(y)(x-y)^j\|}{|x-y|^{m-k}},\quad x,y\in E.
\end{equation}
We  show that full answers to the questions presented in Problem \ref{problem:ext} can be obtained by assuming that the jet $\lbrace A_k \rbrace_{k=0}^m \in  \dot{\op{VJ}}^{m,\omega}(E,V)$ vanishes in an appropriate manner, as described in the next definition.

\begin{dfn}\label{def:scales:jets}
Let $\omega:(0,\infty)\to (0,\infty)$ be a modulus of continuity, $E\subset \R^n$ an arbitrary set, $m \in \N \cup \lbrace 0 \rbrace$, and $V$ a normed space. We define the vanishing subspaces of jets
$$
	\dot{\op{VJ}}_{\op{small}}^{m,\omega}(E,V) := \Big\{   \lbrace A_k \rbrace_{k=0}^m \in  \dot{\op{J}}^{m,\omega}(E,V) \, : \, \lim_{\delta \to 0}  \sup_{ \substack{x\not=y\in E \\ |x-y| \leq \delta}} \frac{R  \left( \lbrace A_k \rbrace_{k=0}^m, x,y \right)}{\omega(|x-y|)}  = 0 \Big \},
$$
$$
		\dot{\op{VJ}}_{\op{large}}^{m,\omega}(E,V) := \Big\{  \lbrace A_k \rbrace_{k=0}^m \in \dot{\op{J}}^{m,\omega}(E,V) \, : \, \lim_{\delta \to \infty}  \sup_{ \substack{x\not=y\in E \\ |x-y| \geq \delta}} \frac{R  \left( \lbrace A_k \rbrace_{k=0}^m, x,y \right)}{\omega(|x-y|)} = 0 \Big \},
$$
$$
	\dot{\op{VJ}}_{\op{far}}^{m,\omega}(E,V) := \Big\{ \lbrace A_k \rbrace_{k=0}^m \in  \dot{\op{J}}^{m,\omega}(E,V)\, : \, \lim_{\delta \to \infty} \sup_{\substack{x\not=y\in E \\ \min(|x|,|y|) \geq \delta}} \frac{R  \left( \lbrace A_k \rbrace_{k=0}^m, x,y \right)}{\omega(|x-y|)}  = 0 \Big\}.
$$
And we define the intersection of the three,
$$
\dot{\op{VJ}}^{m,\omega}(E,V):= \dot{\op{VJ}}_{\op{small}}^{m,\omega}(E,V) \cap \dot{\op{VJ}}_{\op{large}}^{m,\omega}(E,V)  \cap \dot{\op{VJ}}_{\op{far}}^{m,\omega}(E,V) .
$$
\end{dfn}
 
 Before stating our main results, let us finally fix the class of moduli that we consider.
We always assume that $\omega$ is non-decreasing and satisfies 
\begin{equation}\label{eq:mod:intro}
\lim_{t \to 0} \omega(t)=0, \quad \text{and} \quad \frac{s}{\omega(s)} \leq C_\omega \frac{t}{\omega(t)}, \quad \text{whenever} \quad 0 < s \leq t < \infty,
\end{equation} 
 where $C_\omega>0$ is a fixed constant. We also assume the conditions  
 \begin{align}\label{eq:mod:regularity}
	\lim_{t \to 0} \frac{t}{\omega(t)}=0,
\end{align}
 \begin{align}\label{eq:mod:coer}
 \lim_{t\to \infty}\omega(t) = \infty,
\end{align} 
to deal with the small, and large and far scales, respectively.
Now, the following is our main result.
\begin{theorem}\label{thm:main:ext}
Let $\omega$ be a modulus of continuity satisfying \eqref{eq:mod:intro}, $E\subset \R^n$ an arbitrary set, $m \in \N \cup \lbrace 0 \rbrace$, and $V$ a Banach space. For an $m$-jet $\lbrace A_k \rbrace_{k=0}^m \in  \dot{\op{J}}^{m,\omega}(E,V)$, the following hold.

\begin{itemize}
\item Provided that \eqref{eq:mod:regularity} is satisfied, the jet $\lbrace A_k \rbrace_{k=0}^m$ admits an extension $F \in \VC^{m,\omega}_{\op{small}}(\R^n,V)$ if and only if $\lbrace A_k \rbrace_{k=0}^m \in \dot{\op{VJ}}_{\op{small}}^{m,\omega}(E,V).$
\item Provided that \eqref{eq:mod:coer} is satisfied, and $\Gamma\in \{ \op{large}, \op{far} \},$ the jet $\lbrace A_k \rbrace_{k=0}^m$ admits an extension $F \in \VC^{m,\omega}_{\Gamma}(E,V)$ if and only if $\lbrace A_k \rbrace_{k=0}^m \in \dot{\op{VJ}}_{\Gamma}^{m,\omega}(\R^n,V).$
	\item Provided that both \eqref{eq:mod:regularity}, \eqref{eq:mod:coer} hold, the jet $\lbrace A_k \rbrace_{k=0}^m$ admits an extension $F \in \VC^{m,\omega}(\R^n,V)$ if and only if $\lbrace A_k \rbrace_{k=0}^m \in \dot{\op{VC}}^{m,\omega}(E,V).$
\end{itemize} 
Moreover, when $E$ is closed,  the result holds when $V$ is merely a normed space, and these extensions can be defined via the linear Whitney extension operator. 
\end{theorem}

Specifying to the particular case $m=0,$ we obtain the following.
\begin{corollary}\label{cor1}
Let $\omega$ be a modulus of continuity satisfying \eqref{eq:mod:intro}, $E\subset \R^n$ be an arbitrary set  and $V$ a Banach space.
\begin{itemize}
\item Provided that \eqref{eq:mod:regularity} is satisfied, $f\in \dot{C}^{0,\omega}(E,V)$ admits an extension $F \in \VC^{0,\omega}_{\op{small}}(\R^n,V)$ if and only if $f \in \dot{\op{VC}}_{\op{small}}^{0,\omega}(E,V).$
\item Provided that \eqref{eq:mod:coer} is satisfied,  and $\Gamma\in \{ \op{large}, \op{far} \},$ $f\in \dot{C}^{0,\omega}(E,V)$ admits and extension $F \in \VC^{0,\omega}_{\op{\Gamma}}(\R^n,V)$ if and only if $f \in \dot{\op{VC}}_{\Gamma}^{0,\omega}(E,V).$
	\item Provided that both \eqref{eq:mod:regularity}, \eqref{eq:mod:coer} hold, $f\in \dot{C}^{0,\omega}(E,V)$ admits an extension $F \in \VC^{0,\omega}(\R^n,V)$ if and only if $f \in \dot{\op{VC}}^{0,\omega}(E,V).$
\end{itemize} 

\end{corollary}

Combining Corollary \ref{cor1} with the approximations of globally defined functions \cite[Corollary 3.14.]{MudOik24}, we deduce the following.
\begin{corollary} Let $\omega$ satisfy \eqref{eq:mod:intro}, \eqref{eq:mod:regularity} and \eqref{eq:mod:coer}, let $E\subset \R^n$ be closed and $V$ be a Banach space.
	\begin{itemize}
		\item 	If $f\in \dot {\op{VC}}^{0,\omega}_{\op{small}}(E,V)$ and $\varepsilon>0,$ there exists $G\in \Lip(\R^n,V)\cap C^{\infty}(\R^n;V)\cap \dot{C}^{0,\omega}(\R^n,V)$ such that  
		$$
		\| f-G_{|E}\|_{\dot{C}^{0,\omega}(E,V)}<\varepsilon,\qquad \| G\|_{\dot{C}^{0,\omega}(\R^n,V)}\lesssim  \| f \|_{\dot{C}^{0,\omega}(E,V)}.
		$$
		\item  If  $f\in \dot {\op{VC}}^{0,\omega}(E,V)$ and $\varepsilon>0,$ there exists $G\in C^{\infty}_c(\R^n;V)\cap \dot{C}^{0,\omega}(\R^n,V)$ such that  
		$$
		\| f-G_{|E}\|_{\dot{C}^{0,\omega}(E,V)}<\varepsilon,\qquad \| G\|_{\dot{C}^{0,\omega}(\R^n,V)}\lesssim  \| f \|_{\dot{C}^{0,\omega}(E,V)}.
		$$
	\end{itemize} 
\end{corollary}

\subsection*{Basic notation}
Above $C^{\infty}_c$ stands for smooth and compactly supported functions and $\Lip$ stands for Lipschitz continuous functions.
We denote $A \lesssim B$, if $A \leq C B$ for some constant $C>0$ depending only
on parameters like integration exponents or the dimension of $\R^n$ that are not important to track.
Furthermore, we set $A \sim B$, if $A \lesssim B$ and $B \lesssim  A$.
Subscripts or variables on constants and quantifiers, such as, $C_{a},C(a)$ and
$\lesssim_{a}$, signify their dependence on those subscripts.

\section{Necessity of Theorem \ref{thm:main:ext}}

In this section, we give the proof of the ``only if'' parts of Theorem \ref{thm:main:ext}. We begin with the following elementary observation.

\begin{remark}\label{remark:boundedness}
{\em If $\mathcal{A} \in \dot{\op{J}}^{m,\omega}(E,V),$ then for every $r>1$, one has
$$
\max_{0 \leq k \leq m} \sup_{z\in E, \, |z| \leq r}  \|A_k(z)\| \leq C(r,m, \omega, \mathcal{A}, E)< \infty .
$$
}
\end{remark}
\begin{proof}
Write, for each $z\in E \cap B(0,r)$ and $z_0\in E$, 
\begin{align*}
\|A_k(z)\| & \leq \Big \|A_k(z)   -\sum_{j=0}^{m-k} \frac{1}{j!}A_{k+j}(z_0)(z-z_0)^j \Big \| + \| \sum_{j=0}^{m-k} \frac{1}{j!}A_{k+j}(z_0)(z-z_0)^j \Big \| \\
& \leq \| \mathcal{A} \|_{\dot{\op{J}}^{m,\omega}(E,V)} |z-z_0|^{m-k} \omega(|z-z_0|) + C(\mathcal{A},z_0,m) \max_{0 \leq j \leq m} |z-z_0|^j \\
& \lesssim (r+|z_0|)^{m-k} \left( \omega(r+|z_0|) + C(\mathcal{A},z_0,m) \right).
\end{align*}
\end{proof}

The following Lemma \ref{lem:dist=far} is extremely useful both in this and the next section. It shows that the minimum $\min(|x|,|y|) \geq \delta$ defining the vanishing condition $\dot{\op{VJ}}_{\op{far}}^{m,\omega}(E,V)$ can be replaced with the maximum $\max(|x|,|y|) \geq \delta.$ 
We also note that the proof in the case $m=0$ is simpler due to the symmetry of the corresponding $R(\mathcal{A},x,y).$ For a proof in this special case with $E = \R^n,$ we direct the interested reader to \cite[Lemma 2.2.]{MudOik24}.
\begin{lemma}\label{lem:dist=far} Let $\omega$ satisfy \eqref{eq:mod:coer}, $E\subset\R^n$ be arbitrary, $m\in \N \cup \lbrace 0 \rbrace,$ and $V$ a normed space. Then, we have 
	\begin{align*}
			\dot{\op{VJ}}_{\op{far}}^{m,\omega}(E,V)   =  \Big\{ \mathcal{A}  \in  \dot{\op{J}}^{m,\omega}(E,V) \, : \, \lim_{\delta \to \infty} \sup_{ \substack{x\not=y\in E \\ \max(|x|,|y|) \geq \delta}} \frac{R  \left( \mathcal{A}, x,y \right) }{\omega(\aabs{x-y})}    = 0 \Big\}.
	\end{align*}
\end{lemma}
\begin{proof}
 First notice that the vanishing condition of $	\dot{\op{VJ}}_{\op{far}}^{m,\omega}(E,V)$ is \textbf{not} vacuous if and only if $E$ is unbounded, and so $E$ is assumed to be unbounded in this proof. Let $\mathcal{A}=\lbrace A_k \rbrace_{k=0}^m  \in  \dot{\op{VJ}}^{m,\omega}(E,V)$ and $\varepsilon>0.$ Let $M = M(\varepsilon)$ be such that if $u,v\in E$ and $| u|,|v|>M,$ then 
\begin{equation}\label{eq:lem:far=dist:choiceofM}
R  \left( \mathcal{A}, u,v \right)  <\varepsilon\omega(|u-v|).
\end{equation} 
Since $E$ is unbounded, we can fix a point $z_0 \in  E  \setminus B(0,M)$. Now, by Remark \ref{remark:boundedness} and the assumption \eqref{eq:mod:coer}, we can select $K>0$ large enough (and depending on $\varepsilon,M, \omega, E ,\mathcal{A}$) so that 
\begin{equation}\label{eq:lem:far=dist:choiceofK}
K \geq 4 \max \left(     M, |z_0| ,1 \right) \quad \text{and} \quad \max_{0 \leq j \leq m} \sup_{z\in E, \, |z| \leq |z_0|}  \|A_j(z)\| \leq \varepsilon \omega(K).
\end{equation} 
Because $|z_0| \geq M,$ the last bound is also true when $|z_0|$ is replaced by $M.$ (Morally speaking, we take $K\gg M + |z_0|$ to be much bigger than both $M$ and $|z_0|.$)

Now we consider arbitrary $x,y\in E$ with the lower bound $\max ( |x|,| y| ) \geq  K.$ Notice that, for our purposes, we can also assume that $M \geq \min(|x|,|y|),$ as otherwise the desired estimate is true already by \eqref{eq:lem:far=dist:choiceofM}. So, for future reference, we are assuming that:
\begin{equation}\label{eq:lem:far=dist}
\max(|x|, |y|) \geq K \quad \text{and} \quad M \geq \min(|x|,|y|).
\end{equation}
Notice that $R(\mathcal{A},x,y)$ is not symmetric with respect to $(x,y),$ and so we must study two non-symmetric cases.

We start by checking the case $|x| \geq K$ and $|y|\leq M$ first.  Using the first inequality of \eqref{eq:lem:far=dist:choiceofK}, the fact $|z_0| \geq M$, and the triangle inequality, we see that
\begin{equation}\label{eq:lem:far=dist:relations}
|y-z_0| \leq |x-y|  \quad \text{and} \quad C^{-1} |x-y| \leq |x-z_0| \leq C|x-y|;
\end{equation}
for some absolute constant $C>0.$ Now, using that $x,z_0\in E \setminus B(0,M)$, the relations \eqref{eq:lem:far=dist:relations} and the choice of $K$ in \eqref{eq:lem:far=dist:choiceofK}, we have, for every $0 \leq k \leq m,$
	\begin{align}\label{eq:p}
		\Big \|A_k(x) & -\sum_{j=0}^{m-k} \frac{1}{j!}A_{k+j}(y)(x-y)^j \Big \|  \nonumber \\
		& \leq \Big \|A_k(x)-\sum_{j=0}^{m-k} \frac{1}{j!}A_{k+j}(z_0)(x-z_0)^j \Big \|+ \Big \|  \sum_{j=0}^{m-k} \frac{1}{j!}A_{k+j}(z_0)\left( (y-z_0)^j-(x-z_0)^j \right) \Big \| \nonumber \\ 
		 & \leq \varepsilon  |x-z_0|^{m-k} \, \omega(|x-z_0|) + C(m) \max_{0 \leq j \leq m-k} \sup_{z\in E, \, |z| \leq |z_0|} \| A_{k+j}(z) \| \left( |y-z_0|^j + |x-z_0|^j \right), \nonumber \\
		 & \lesssim_{C_\omega,m,E,\|\mathcal{A}\|_{\dot{\op{VJ}}^{m,\omega}(E,V)}} \varepsilon |x-y|^{m-k} \omega(|x-y|) + \varepsilon \omega(K) \max_{0 \leq j \leq m-k}  |x-y|^j .
	\end{align}
 Since $|x-y| \geq K- M \geq K/2 \geq 1$ by \eqref{eq:lem:far=dist:choiceofK}, we have that $w(K) \lesssim_\omega \omega(|x-y|)$ (see condition \eqref{eq:mod:intro} of $\omega$) and also $|x-y|^j \leq |x-y|^{m-k}$ for all $0 \leq j \leq m-k.$ These applied to \eqref{eq:p} yield the bound
\begin{equation}\label{eq:lem:dist=far:boundsx-y}
 \Big \|A_k(x)   -\sum_{j=0}^{m-k} \frac{1}{j!}A_{k+j}(y)(x-y)^j \Big \|    \lesssim \varepsilon  |x-y|^{m-k} \omega(|x-y|) , \quad k=0,\ldots, m.
\end{equation}
Dividing by the term $|x-y|^{m-k} \omega(|x-y|)$, and taking the maximum among $k=0,\ldots,m,$ we arrive at $R(\mathcal{A},x,y)/\omega(|x-y|) \lesssim \varepsilon ,$ which is the desired estimate.
	
\smallskip

Let us now study the other (non-symmetric) case of \eqref{eq:lem:far=dist}, that is, when $|y| \geq K,$ and $|x|\leq M.$ Let us write, for every $0 \leq k \leq m,$
\begin{align}\label{eq:lem:far=dist:chainestimates}
 &\Big \|A_k(x)   -\sum_{j=0}^{m-k} \frac{1}{j!}A_{k+j}(y)(x-y)^j \Big \| \nonumber \\
& \leq \Big \|\sum_{j=0}^{m-k} \frac{1}{j!} \Big( A_{k+j}(y) - \sum_{l=0}^{m-k-j} \frac{1}{l!}A_{k+j+l}(x)(y-x)^l \Big) (x-y)^j \Big \| \nonumber \\
& \qquad  + \Big \| A_k(x) - \sum_{j=0}^{m-k} \frac{1}{j!} \sum_{l=0}^{m-k-j} \frac{1}{l!} A_{k+j+l}(x)(y-x)^l   (x-y)^j \Big \| \nonumber  \\
& \lesssim_m \sum_{j=0}^{m-k}  \Big \|A_{k+j}(y) - \sum_{l=0}^{m-k-j} \frac{1}{l!}A_{k+j+l}(x)(y-x)^l \Big \| |x-y|^j +      \max_{1 \leq l + j \leq m-k} \|A_{k+j+l}(x)\||x-y|^{l + j}  \nonumber \\
&   \lesssim_m \sum_{j=0}^{m-k}  \Big \|A_{k+j}(y) - \sum_{l=0}^{m-k-j} \frac{1}{l!}A_{k+j+l}(x)(y-x)^l \Big \| |x-y|^j +     \max_{1 \leq j \leq m } \|A_{j}(x)\||x-y|^{m-k} ;
\end{align}
where in the last estimate we used that $|x-y| \geq K-M \geq 1$ by virtue of \eqref{eq:lem:far=dist:choiceofK}, and so $|x-y|^{l+j} \leq |x-y|^{m-k}$ whenever $1 \leq l+ j \leq m-k$. Now, for the first term in \eqref{eq:lem:far=dist:chainestimates}, we use the estimates \eqref{eq:lem:dist=far:boundsx-y} for each $j=0,\ldots,m-k$, but swapping the roles of $x$ and $y$ (note that the estimates of \eqref{eq:lem:dist=far:boundsx-y} were obtained for $|x| \geq K$ and $|y| \leq   M$). We thus obtain, 
\begin{align*}
\sum_{j=0}^{m-k}  & \Big \|A_{k+j}(y) - \sum_{l=0}^{m-k-j} \frac{1}{l!}A_{k+j+l}(x)(y-x)^l \Big \| |x-y|^j \\
& \lesssim \sum_{j=0}^{m-k}  \varepsilon |x-y|^{m-k-j} \omega(|x-y|) |x-y|^j \lesssim_m   \varepsilon |x-y|^{m-k} \omega(|x-y|).
\end{align*}
As for the second term in \eqref{eq:lem:far=dist:chainestimates}, we use the current assumption $|x| \leq M$ and the second inequality of \eqref{eq:lem:far=dist:choiceofK} to infer that 
$$ 
\max_{1 \leq j \leq m } \|A_{j}(x)\||x-y|^{m-k}\leq \max_{1 \leq j \leq m} \sup_{u \in E, \, |u| \leq M } \|A_j(u)\||x-y|^{m-k} \leq  \varepsilon \omega(K)|x-y|^{m-k}.
$$
But again $|x-y| \geq K-M \geq K/2$ thanks to \eqref{eq:lem:far=dist:choiceofK}, and so property \eqref{eq:mod:intro} of $\omega$ tells us that $\omega(K) \lesssim_\omega \omega(|x-y|).$

Dividing by $|x-y|^{m-k} \omega(|x-y|)$ in \eqref{eq:lem:far=dist:chainestimates} and combining the above two estimates, we conclude that
$$
\frac{R(\mathcal{A},x,y)}{\omega(|x-y|)} =  \max_{0 \leq k \leq m } \frac{\|A_k(x)  -\sum_{j=0}^{m-k} \frac{1}{j!}A_{k+j}(y)(x-y)^j  \| }{|x-y|^{m-k} \omega(|x-y|)} \lesssim \varepsilon,
$$
as desired. 
\end{proof}

So, let us assume that $F \in \VC^{m,\omega}_{\Gamma}(\R^n,V)$ and we will next prove that the restriction of $F$ to $E$ satisfies the properties from Definition \ref{def:scales:jets}. While the proof in the case $\Gamma=\op{small}$ is immediate from Taylor's theorem, for the scales $\Gamma = \op{large}$ and $\Gamma = \op{far}$, we need to study a couple of subcases separately.

\subsection*{The case $\Gamma= \op{small}$} Because $F \in \dot{C}^{m,\omega}(\R^n,V)$ we use Taylor's theorem to write, for each couple $x,y\in \R^n$ of distinct points, and each $k=0,\ldots,m$:
\begin{equation}\label{eq:Taylor:argument}
 \frac{\|D^k F(x)-\sum_{j=0}^{m-k} \frac{1}{j!}D^{k+j} F (y)(x-y)^j\|}{\omega(|x-y|)|x-y|^{m-k}} \leq \frac{1}{(m-k)!} \sup_{z\in [x,y]} \frac{\|D^m F(z)-D^m F(y)\|}{\omega(|x-y|)}.
\end{equation}
Let us briefly remind why, in the non-trivial case $0\leq k \leq m-1$, estimate \eqref{eq:Taylor:argument} holds in the setting of normed-valued $C^{m,\omega}$ mappings. Indeed, for each linear and continuous functional $v^*\in V^*,$ with $\|v^*\|_*=1$ (meaning the dual norm of $V$), we have that $v^* \circ F \in C^m( \R^n, \R)$ with 
$$
D^j( v^* \circ  F )(x) = v^* \circ  D^j F (x) , \quad x\in \R^m, \, j =1, \ldots, m.
$$
Thus, using the notation
$$
\| v^* ( A) \| = \sup\lbrace \| v^*(A(u_1,\ldots,u_k)) \, : \,  u_1, \ldots, u_k \in \mathbb{S}^{n-1} \rbrace , 
$$
for every $k$-linear mapping $A$ from $\R^n$ to $V,$ and defining
$$
A_{x,y,k}: =  D^k F(x)-\sum_{j=0}^{m-k} \frac{1}{j!}D^{k+j} F (y)(x-y)^j, \quad x,y\in \R^n, \, x \neq y, \, k=0,\ldots,m-1,  
$$
Taylor's theorem applied to $v^* \circ F$ gives us, for all $k=0,\ldots,m-1,$
\begin{align*}
  & \frac{  \| v^* \left(   A_{x,y,k} \right)   \|}{\omega(|x-y|)|x-y|^{m-k}}  = \frac{\Big \|   D^k (v^* \circ F)(x)-\sum_{j=0}^{m-k} \frac{1}{j!} D^{k+j} (v^* \circ F ) (y)(x-y)^j  \Big \| }{\omega(|x-y|)|x-y|^{m-k} } \\[0.5em]
 &  \quad \leq   \frac{1}{(m-k)!}   \sup_{z\in [x,y]} \frac{\|D^m (v^* \circ F) (z)-D^m (v^* \circ F)(y)\|}{ \omega(|x-y|)} \\
 &  \quad =  \frac{1}{(m-k)!}  \sup_{z\in [x,y]} \frac{\|v^* \circ D^m  F (z)-v^* \circ D^m F(y)\|}{ \omega(|x-y|)} \leq  \frac{1}{(m-k)!}  \sup_{z\in [x,y]} \frac{\|D^m F(z)-D^m F(y)\|}{\omega(|x-y|)}.
\end{align*}
The Hahn-Banach Theorem in $V$ provided us, for each $k,x,y$ as above, with $v^*\in V^*$ so that $\|v^*\|_*=1$ and $\|v^*(A_{x,y,k})\| = \| A_{x,y,k}\|.$ Together with the previous estimates, this yields \eqref{eq:Taylor:argument}.

One can continue the estimate \eqref{eq:Taylor:argument} by using that $D^m F \in \dot{C}^{0,\omega}(\R^n,V)$, thus deducing that the resctriction of $\lbrace F, DF, \ldots, D^m F \rbrace$ to $E$ defines an $m$-jet in $\dot{\op{J}}^{m,\omega}(E,V)$. To show that those restrictions actually belong to $\dot{\op{VJ}}_{\op{small}}^{m,\omega}(E,V)$ is also very easy: because $F\in \VC^{m,\omega}_{\op{small}}(\R^n,V),$ given $\varepsilon>0$ there exists $\delta>0$ so that 
$$
\|D^m F(u)-D^m F(v)\| \leq \varepsilon \omega(|u-v|), \quad |u-v| \leq \delta.
$$
Thus, assuming that $|x-y| \leq \delta,$ and as obviously then $|z-y| \leq \delta$ for each $z\in [x,y],$ the right hand side of \eqref{eq:Taylor:argument} is bounded from above by $C(m) \varepsilon.$

 \subsection*{The case $\Gamma= \op{large}$} Here we assume \eqref{eq:mod:coer} for $\omega$ and that $F\in \dot{\op{VC}}_{\op{large}}^{m,\omega}(\R^n,V)$, and estimate \eqref{eq:Taylor:argument} in the following manner. For any $\varepsilon>0,$ let $R>0$ so that $|u-v| \geq R $ implies $ \|D^m F(u)-D^m F(v)\| \leq \varepsilon \omega(|u-v|).$ Let us assume that $|x-y| \geq M, $ where $M \gg R$ and its value will be specified in a moment. For those $z\in [x,y]$ so that $|z-y|\geq R,$ it is enough to write
 $$
 \frac{\|D^m F(z)-D^m F(y)\|}{\omega(|x-y|)} \leq \varepsilon \frac{\omega(|z-y|)}{\omega(|x-y|)} \leq \varepsilon. 
 $$
For those $z\in [x,y]$ with $|z-y| <  R, $ we see that
$$
 \frac{\|D^m F(z)-D^m F(y)\|}{\omega(|x-y|)} \leq \| F\|_{\dot C^{m,\omega}(\R^n,V)}\frac{\omega(|z-y|)}{\omega(|x-y|)} \leq \| F\|_{\dot C^{m,\omega}(\R^n,V)}\frac{\omega(R)}{\omega(|x-y|)} .
$$
Due to the assumption \eqref{eq:mod:coer}, if $M \gg R$ is large enough, then $|x-y| \geq M$ implies that the last term can be made smaller than $\varepsilon.$ This shows that the jet given by the restriction of $\lbrace F, DF, \ldots, D^m F \rbrace$ to $E$ belongs to $\dot{\op{VJ}}_{\op{large}}^{m,\omega}(E,V).$

\subsection*{The case $\Gamma= \op{far}$} Again we assume \eqref{eq:mod:coer} for $\omega$ and that $F\in \VC_{\op{far}}^{m,\omega}(\R^n,V).$ Applying Lemma \ref{lem:dist=far} for $D^m F \in \VC_{\op{far}}^{0,\omega}(\R^n,\mathcal{L}^m(\R^n,V))$, it follows that for every $\varepsilon>0$ there exists $R>0$ so that if $u,v\in \R^n$ with $|u| \geq R $, then
\begin{align}\label{eq:applydistfar}
	 \frac{\|D^m F(u)-D^m F(v)\|}{\omega(|u-v|)} \leq \varepsilon. 
\end{align}
Also, by the condition \eqref{eq:mod:coer}, we can find $M>0$ (depending on $\varepsilon$, $R$ and $\omega$) so that
\begin{equation}\label{eq:proofnecessityfar_choiceM}
M \geq 2 R \quad \text{and} \quad \omega(2R) \leq \varepsilon \omega(M/2).
\end{equation}
Now, let $x,y \in \R^n$ be such that $|x|, |y| \geq M$. By the continuity of $D^mF,$ we can find a point $z\in [x,y]$ maximizing the supremum of the right-hand side of \eqref{eq:Taylor:argument}. In the case where $\max( |y|, |z| ) \geq R,$ then 
$$
\|D^m F(z)-D^m F(y)\| \leq \varepsilon \omega(|y-z|) \leq \varepsilon \omega(|x-y|),
$$
by the bound \eqref{eq:applydistfar}. And when $|y|, |z| < R,$ we can estimate as in the case $\Gamma=\op{large}$:
$$
\frac{\|D^m F(z)-D^m F(y)\|}{\omega(|x-y|)} \leq \| F\|_{\dot C^{m,\omega}(\R^n,V)}\frac{\omega(|z-y|)}{\omega(|x-y|)} \leq \| F\|_{\dot C^{m,\omega}(\R^n,V)}\frac{\omega(2R)}{\omega(|x-y|)} .
$$
 But observe that, by the triangle inequality and \eqref{eq:proofnecessityfar_choiceM}, we get $|x-y| \geq M- R \geq M/2.$ In combination with the second inequality of \eqref{eq:proofnecessityfar_choiceM}, we infer that the last term is bounded above by
$$
\| F\|_{\dot C^{m,\omega}(\R^n,V)}\frac{\omega(2R)}{\omega(|x-y|)} \leq \| F\|_{\dot C^{m,\omega}(\R^n,V)} \frac{\omega(2R)}{\omega(M/2)} \leq \| F\|_{\dot C^{m,\omega}(\R^n,V)} \varepsilon \lesssim \varepsilon.
$$    
Thus, we have the desired estimate for all possible $z\in [x,y],$ and so \eqref{eq:Taylor:argument} gives that the restriction of $\lbrace F, DF, \ldots, D^m F \rbrace$ to $E$ belongs to $\dot{\op{VJ}}_{\op{far}}^{m,\omega}(E,V).$
\section{Sufficiency of Theorem \ref{thm:main:ext}}

This section is devoted to the ``if parts'' of Theorem \ref{thm:main:ext}. We extend a jet from $\dot{\op{J}}^{m,\omega}(E,V)$ to a $\dot C^{m,\omega}(\R^n,V)$ function while simultaneously preserving a given vanishing scale. As before, our modulus $\omega :(0, \infty) \to (0, \infty)$ is non-decreasing, with $\lim_{t\to 0}\omega(t) =0$, and such that
\begin{equation}\label{eq:mod:ext1}
\frac{s}{\omega(s)} \leq C_\omega \frac{t}{\omega(t)}, \quad \text{whenever} \quad 0 < s \leq t < \infty,
\end{equation}
for some  $C_\omega >0;$
or equivalently that
\begin{equation}\label{eq:mod:ext2}
\omega(\lambda t ) \leq C_\omega \lambda \omega(t), \quad \text{for all} \quad \lambda \geq 1,\, t > 0.
\end{equation}

The first reduction in the extension Problem \ref{problem:ext} is to notice that the unique continuous extension $\bar{\mathcal{A}}$ of $\mathcal{A} \in \dot{\op{VJ}}^{m,\omega}_{\Gamma}(E,V)$ to the closure $\bar E$ of $E$ satisfies
\begin{align}\label{eq:JETclose}
	\|\bar{\mathcal{A}}\|_{\dot{\op{J}}^{m,\omega}(\overline{E},V)}=\|\mathcal{A}\|_{\dot{\op{J}}^{m,\omega}(E,V)},\quad \text{and} \quad \bar{\mathcal{A}}\in \dot{\op{VJ}}^{m,\omega}_{\Gamma}(\bar{E},V),
\end{align}
for each scale $\Gamma\in \{\op{small}, \op{far},\op{large}\}.$ 

Indeed, this is possible since we are assuming that $V$ is a Banach space, and thus the functions $A_k\in (A_1,\ldots,A_m)\in \mathcal{A}$ map Cauchy sequences of $E$ to Cauchy sequences of $\calL^k(\R^n,V).$

Since $V$ is Banach, so is $\calL^m(\R^n,V).$ Thus if $\{ x_j\}$ is Cauchy in $E$ with a limit $\bar{x}\in \bar{E},$ then $\{A_m(x_j)\}$ is Cauchy in $\calL^m(\R^n,V),$ by the estimate 
\[
\| A_m(x_j) - A_m(x_i)  \| \lesssim \| \mathcal{A} \|_{\dot{\op{J}}^{m,\omega}(E,V)} \omega(|x_j-x_i|)\to 0,\quad \min(i,j)\to\infty.
\] 
Thus the limit $\bar A_m(\bar{x}):=A_m(\bar{x})\in \calL^m(\R^n,V)$ is uniquely determined and in this way we extend $A_m$ to $\bar{E}.$ Now the extension of the case $k=m-1,$ i.e. $A_{m-1}$ as $\bar A_{m-1},$ follows by recursing from the case $k=m$ above and the definition of \eqref{eq:norm:jet}.
Thus for the extension problem, we can assume that $E \subset \R^n$ is closed. In fact, only here we need the completeness of $V.$ Thus, if the given set $E \subset \R^n$ is closed, Theorem \ref{thm:main:ext} holds for $V$ a normed space.

We next describe the linear extension operator defined by the classical Whitney decomposition and the associated partition of unity. We remark that even if we restrict ourselves to the particular case $m=0$ and $V=\R,$ any simple variation of the infimal convolution formula, as mentioned in the introduction, does not seem to provide an extension operator that maps $ \dot {\op{VC}}^{0,\omega}_{\Gamma}(E,\R)\to\dot {\op{VC}}^{0,\omega}_{\Gamma}(\R^n,\R),$ for any of the scales $\Gamma \in \{\op{small}, \op{far},\op{large}\}.$  Moreover, it is an obvious benefit that the extension operator is linear, and of course any operator defined through an infimal convolution formula is automatically non-linear.

\smallskip

\subsubsection*{Whitney partition of unity}\label{section:sufficienty}

We begin by recalling the main properties of the Whitney decomposition of an open set into cubes.
For a closed set $E,$ the Whitney cubes associated with the open set $\R^n \setminus E$ is a collection $ \mathcal{Q} $ of dyadic cubes with the following properties:
\begin{itemize}
\item  There holds that $\bigcup_{Q \in \mathcal{Q}} Q = \R^n \setminus E.$
\item  For every $Q \in \mathcal{Q},$ there holds that $d(Q,E) \leq \diam(Q) \leq 4 d(Q,E).$
\item  If $Q, Q' \in \mathcal{Q}$ are two distinct cubes, then $\mathrm{int}(Q) \cap \mathrm{int}(Q') = \emptyset.$
\end{itemize}
To construct a $C^\infty$ partition of unity associated with these cubes, for each $Q \in \mathcal{Q}$ denote $Q^*:= 9/8Q$ and let $p_Q\in E$ be a point (not necessarily unique) for which $d(E,Q)=d(p_Q,Q).$ A $C^\infty$-Whitney partition of unity is, in particular, a collection of functions $\lbrace \varphi_Q \, : \, Q\in \mathcal{Q} \rbrace$ such that each $\varphi_Q$ is supported on $Q^*.$ Many relevant properties hold for these families of cubes and functions, and we refer to \cite[Chapter VI]{Stein1970} for a detailed exposition of this topic. 

\begin{propo}\label{prop:WhitneyPartition} The Whitney partition of unity satisfies the following properties.
\begin{enumerate}[label=\normalfont{(\roman*)}]
\item\label{Wh:totalunion}  There holds that 
$$
\bigcup_{Q \in \mathcal{Q}} Q =\bigcup_{Q \in \mathcal{Q}} Q^* = \R^n \setminus E.
$$
\item\label{Wh:bdoverlapp} There exists a dimensional constant $N(n)>0$ so that
\[
\sum_{Q\in\calQ}1_{Q^*}\leq N(n);
\]
i.e. every $x\in \R^n  \setminus E$ is contained in a neighbourhood that intersects at most $N(n)$ cubes of the family $\lbrace Q^* \, : \, Q\in \mathcal{Q} \rbrace.$ 
\item\label{Wh:estimatedistancecubes} There exists an absolute constant $C>0$ so that for all cubes $Q\in \mathcal{Q}$ there holds that 
$$
|p_Q-y| \leq C |x-y|, \quad \text{whenever} \quad x\in Q^*,\, y\in E.
$$

\item\label{Wh:smoothpartition} For each $Q\in\calQ,$ there holds that 
$$
0\leq \varphi_Q \leq 1_{Q^*},\qquad \varphi_Q \in C^\infty(\R^n).
$$
\item\label{Wh:sumofpartition1} Partition of unity:  for all $x\in \R^n\setminus E,$ there holds that 
$$\sum_{Q \in \mathcal{Q}} \varphi_Q(x)=\sum_{Q  \, : \,  Q^*\ni x} \varphi_Q(x)=1,\qquad \sum_{Q \in \mathcal{Q}} D^k \varphi_Q(x)=0, \quad k \in \N.
$$ 
Notice that the latter follows from the first and \ref{Wh:bdoverlapp}.
\item\label{Wh:estimategradients} For each $k \in \N,$ there exists a constant $C(n,k)$ so that for all $Q\in  \mathcal{Q}$ and  $z\in Q^*,$ there holds that  
$$
\| D^k \varphi_Q(z) \| \leq C(n,k) d(z,E)^{-k}.
$$ 
\end{enumerate}
\end{propo}

\medskip

Now, given a jet $\mathcal{A}=\lbrace A_k \rbrace_{k=0}^m \in \dot{\op{J}}^{m,\omega}(E,V)$ and a point $y\in E,$ we define the polynomial
$$
P_y^{\mathcal{A}} : \R^n \to V,\qquad P_y^{\mathcal{A}}(x)= \sum_{k=0}^m \frac{1}{k!} A_k(y)(x-y)^k, \quad x\in \R^n .
$$
And with these, the Whitney extension operator 

\begin{equation}\label{eq:WhitneyExt1}
	\begin{split}
		\calW (\mathcal{A} ) (x) =  \begin{cases}
			A_0(x),\quad &  x\in E,\\
			\sum\limits_{Q\in \mathcal{Q}} \varphi_Q(x) P_{p_Q}^{\mathcal{A}}(x) , \quad \quad &x\in \mathbb{R}^n\setminus E.
		\end{cases} 
	\end{split}
\end{equation}
Using the properties \ref{Wh:bdoverlapp} and \ref{Wh:smoothpartition}, it is easy to see that $\calW f \in C^\infty(\R^n \setminus E),$ and for every multi-index $\alpha \in ( \N \cup \lbrace 0 \rbrace )^n$, the $\alpha$-partial derivative $D^\alpha (\calW f) : \R^n \setminus E \to V$ is given by 
\begin{equation}\label{eq:formulaLeibniz}
D^\alpha(\calW (\mathcal{A}))(x)   =\sum_{Q \in \mathcal{Q}} \sum_{\beta \leq \alpha} \binom{\alpha}{\beta} D^\beta \varphi_Q(x)    D^{\alpha-\beta}P_{p_Q}^{\mathcal{A}}(x), \quad  x\in \R^n \setminus E.
\end{equation}
Naturally, for any function $ G: \R^n \to V,$ the $\beta$-partial derivative $D^\beta G (x)$ at $x\in \R^n$ is defined as
$$
D^\beta G(x)= \frac{\partial^{|\beta|} G}{\partial x_1^{\beta_1} \cdots \partial x_n^{\beta_n}}(x) =  D^{|\beta|} G(x) \big ( \overbrace{e_1, \ldots , e_1}^{\beta_1 \text{ times}}, \ldots, \overbrace{e_n, \ldots , e_n}^{\beta_n \text{ times}} \big ), 
$$
where $e_i$ is the $i$th unit vector.

\smallskip
A classical result is that for a $V=\R$ valued $m$-jet $\mathcal{A} \in  \dot{\op{J}}^{m,\omega}(E,V),$ there holds that 
\begin{align}\label{eq:WhitneyExt2}
	\calW (\mathcal{A}) \in \dot C^{m, \omega}(\R^n,V),\qquad \|\calW (\mathcal{A})\|_{\dot C^{m,\omega}(\R^n,V)} \leq \kappa(n, m ,C_\omega) \|\mathcal{A}\|_{\dot{\op{J}}^{m,\omega}(E,V)},
\end{align}
for a constant $\kappa(n,m,C_\omega)$ depending only on $n$, $m,$ and the constant $C_\omega$ from \eqref{eq:mod:ext1}. A proof can be found in \cite[Chapter VI, p.175]{Stein1970}. However, a small modification of the arguments therein shows that
\eqref{eq:WhitneyExt2} is satisfied 
for all $\mathcal{A}\in    \dot{\op{J}}^{m, \omega}(E,V),$ with an arbitrary normed space $V.$ 
\smallskip

Nevertheless, it seems very far from elementary to determine the action of the Whitney extension operator \eqref{eq:WhitneyExt1} over the three vanishing scales.
For each scale, we study pairs $(x,y)$ and their relative position with respect to $E$ and $|x-y|.$ Although this separation into cases will be the same for each scale, we split the proof of Theorem \ref{thm:main:ext} into three parts, with each of the scales presenting its own particular difficulties.

During the proof we will fix an $m$-jet $\mathcal{A} \in  \dot{\op{VJ}}^{m,\omega}_{\Gamma}(E,V) $, whose norm $\|\mathcal{A}\|_{\dot{\op{J}}^{m,\omega}(E,V)}<\infty$ is considered to be an \textit{absolute} constant. 
		Moreover, we indicate the dependence on the dimension or on the order of smoothness $m$ in some estimates by including subscripts, e.g. $\lesssim_n,$ $\lesssim_m,$ $ \lesssim_{n,m}.$ The same applies for the constant $C_\omega$ associated with the modulus $\omega$ from \eqref{eq:mod:ext1}. Thus, when using the notation $ A \lesssim B,$ the constant $C$ involved in the estimate $A \leq C  \cdot B$ is allowed to depend on $n,$ $m,$ $\|\mathcal{A}\|_{\dot{\op{J}}^{m,\omega}(E,V)}$ and $C_\omega.$ 

To simplify notation, we denote by $F:=  \mathcal{W}(\mathcal{A})$ the Whitney Extension \eqref{eq:WhitneyExt1} and also $P_{y}=P_y^{\mathcal{A}},$ for each $y\in E.$

We will use several times the estimates of the following Lemma \ref{lem:derivatives:higherorder}, which are implicitly proved in \cite{Stein1970}.

\begin{lemma}\label{lem:derivatives:higherorder}
 Let $y\in E.$ Let $x\in \R^n\setminus E$ and $\xi = \xi(x) \in E$ be a point that minimizes $\dist(x,E)=|x-\xi|.$ Then,
\begin{enumerate}[label=\normalfont{(\alph*)}]
\item \label{lemmaestimates:partaapproximationDmF} $
\|D^m F(x)- A_m(y)\| \lesssim_{m,n} \sum_{Q^* \ni x} \Big( R (\mathcal{A}, y,p_Q)   +  R (\mathcal{A},\xi,p_Q)  \Big ),
$
\smallskip
\item \label{lemmaestimates:partbestimateDmplus1F} $ \|D^{m+1} F(x)\| \lesssim_{m,n} \sum_{Q^* \ni x}R (\mathcal{A},\xi,p_Q) |x-\xi|^{-1}.$ 
\end{enumerate}
\end{lemma}
\begin{proof}
\item[(a)] Using multi-index notation, the polynomials $P_y,$ $y\in E,$ can be expressed as
\begin{equation}\label{eq:alternate:expression:poly}
P_y(x) := \sum_{k=0}^m \frac{1}{k!} A_k(y)(x-y)^k = \sum_{|\beta| \leq m} \frac{1}{\beta!} A_{\beta}(y)(x-y)^\beta,
\end{equation}
where the $A_{\beta}$s are defined as follows. For any multi-index $\beta = (\beta_1,\dots,\beta_n)\in (\N\cup\{0\})^n$ of order $|\beta| := \sum_{j=1}^n \beta_j \leq m,$ and $y\in E,$ we set 
	\begin{equation}\label{eq:Abetaexpression}
		 A_\beta(y) := A_{|\beta|}(y)(\overbrace{e_1, \ldots, e_1}^{\beta_1 \text{ times}}, \ldots , \overbrace{e_n, \ldots, e_n}^{\beta_n \text{ times}});
		\end{equation}
where $e_i$ is the $i$th unit vector. Also, for any such $\beta$ and $z=(z_1,\ldots,z_n) \in \R^n,$ we denote
$$
z^\beta= z_1^{\beta_1} \cdots z_n^{\beta_n}.
$$ 
The second identity in \eqref{eq:alternate:expression:poly} follows from the symmetry and the $k$-linearity of the $A_k$s' and a basic combinatorial argument. 
Thus, for every multi-index $\alpha \in (\N\cup\{0\})^n$ of order $|\alpha| =  m,$ one has $D^\alpha P_y(x)= A_\alpha(y),$ for every $x\in \R^n$ and $y\in E.$
Now, to show part \ref{lemmaestimates:partaapproximationDmF} it is enough to show that for each multi-index $\alpha$ of order $|\alpha| = m,$ there holds that 
\begin{align}
	\|D^\alpha F(x)- A_\alpha(y)\| \lesssim_{m,n} \sum_{Q^* \ni x} \Big( R (\mathcal{A}, y,p_Q)   +  R (\mathcal{A},\xi,p_Q)  \Big ).
\end{align}
Then, using the multilinearity of $D^mF(x)-A_m(y),$ the bound \ref{lemmaestimates:partaapproximationDmF} follows.

 Also, by property \ref{Wh:sumofpartition1}, $\sum_{Q} D^\gamma \varphi_Q(x)=0$ for any multi-index $\gamma$ with $|\gamma|\geq 1.$ These observations and formula \eqref{eq:formulaLeibniz} permit to write, for $x\in \R^n \setminus E$, $y\in E$, and for a multi-index $\alpha$ of order $|\alpha| = m:$
\begin{align*}
	D^\alpha F(x)-A_\alpha(y) & = \sum_{Q } D^\alpha (\varphi_Q \cdot P_{p_Q} )(x)-A_\alpha(y) = \sum_{Q} \varphi_Q(x) \left( A_\alpha(p_Q)-A_\alpha(y) \right) + \\
	& \quad \: + \sum_{Q} \sum_{\beta \leq \alpha, \, \beta \neq \alpha} \binom{\alpha}{\beta} D^{\alpha-\beta}\varphi_Q(x) \left( D^\beta P_{p_Q}(x)- D^\beta P_\xi(x) \right).
\end{align*} 
The first term is estimated using the definition of $R (\mathcal{A}, \cdot , \cdot)$ \eqref{eq:vanishing:remainders}:
\begin{align*}
\Big \| \sum_{Q} \varphi_Q(x) \left( A_\alpha(p_Q)-A_\alpha(y) \right) \Big \| &  \leq \sum_{Q^* \ni x}  \|A_\alpha(p_Q)-A_\alpha(y)\| \leq  \sum_{Q^* \ni x}  R (\mathcal{A},y,p_Q)  .
\end{align*}

For the second term, we take into account that the polynomials $P_{p_Q}- P_\xi$ have degree up to $m,$ and so, for $|\beta| \leq m,$ 
	 \begin{align*}
D^\beta P_{p_Q}(x)- D^\beta P_\xi(x) & = \sum_{|\gamma| \leq m-|\beta|} \frac{1}{\gamma!} \left[ D^{\beta+\gamma}(P_{p_Q}-  P_\xi)(\xi) \right](x-\xi)^\gamma \\ 
& = \sum_{|\gamma| \leq m-|\beta|} \frac{1}{\gamma!} \left[ D^{\beta+\gamma}P_{p_Q}(\xi)-  A_{\beta+\gamma}(\xi) \right](x-\xi)^\gamma \\
& =  \sum_{|\gamma| \leq m-|\beta|} \frac{1}{\gamma!} \left[ \sum_{|\delta| \leq m-|\beta+ \gamma|} \frac{1}{\delta!} A_{\delta+\beta+\gamma}(p_Q)(\xi-p_Q)^\delta-  A_{\beta+\gamma}(\xi) \right](x-\xi)^\gamma.
\end{align*}
Now from the formula \eqref{eq:vanishing:remainders}, it follows that 
\begin{equation}\label{eq:estimates:deriv:poly}
\|D^\beta P_{p_Q}(x)- D^\beta P_\xi(x)\| \leq \sum_{j =0}^{m-|\beta|}  R (\mathcal{A},\xi,p_Q)  |\xi-p_Q|^{m- |\beta|-j } |x-\xi|^{j}.
\end{equation}
This estimate and property \ref{Wh:estimategradients}, and also $|\alpha| = m,$ lead us to
\begin{align}\label{eq:conclusion:(a)}
\Big  \|   \sum_{Q} & \sum_{\beta \leq \alpha,   \, \beta \neq \alpha} \binom{\alpha}{\beta}   D^{\alpha-\beta}\varphi_Q(x) \left( D^\beta P_{p_Q}(x)- D^\beta P_\xi(x) \right) \Big \| \nonumber \\
& \quad \leq C(n,m) \sum_{Q^* \ni x} \sum_{\beta \leq \alpha,   \, \beta \neq \alpha}  |x-\xi|^{-(m-|\beta|)} \| D^\beta P_{p_Q}(x)- D^\beta P_\xi(x)\| \nonumber \\
&  \quad \leq C(n,m  ) \sum_{Q^* \ni x} R (\mathcal{A},\xi,p_Q)  \sum_{k=0}^{m-1} \sum_{j=0}^{m-k} \left(\frac{|\xi-p_Q|}{|x-\xi|}\right)^{m-k-j}\leq C(n,m) \sum_{Q^* \ni x} R (\mathcal{A},\xi,p_Q),
\end{align}
where in the last bound we used that $|\xi-p_Q| \leq C |x-\xi|,$ for an absolute constant $C,$ by \ref{Wh:estimatedistancecubes}.

\smallskip  

\item[(b)] The proof is very similar to that of \ref{lemmaestimates:partaapproximationDmF}. Since the polynomials have degree up to $m,$ if $\alpha$ is a multi-index with order $|\alpha|=m+1,$ then
\begin{align*}
D^{\alpha} F(x) & =\sum_{Q} \sum_{\beta \leq \alpha, |\beta| \leq m} \binom{\alpha}{\beta} D^{\alpha-\beta}\varphi_Q(x) D^\beta P_{p_Q}(x) \\
& = \sum_{Q} \sum_{\beta \leq \alpha, |\beta| \leq m} \binom{\alpha}{\beta} D^{\alpha-\beta}\varphi_Q(x)  \left( D^\beta P_{p_Q}(x)- D^\beta P_\xi(x) \right).
\end{align*}
By the estimate \eqref{eq:estimates:deriv:poly}, the arguments that led us to \eqref{eq:conclusion:(a)} and to the conclusion of \ref{lemmaestimates:partaapproximationDmF}, we derive
\begin{align*}
\| D^{\alpha} F(x)  \| & \leq C(n,m) \sum_{Q^* \ni x} \sum_{k=0}^m \sum_{\substack{|\beta| = k}}  |x-\xi|^{-(m+1-k)} \| D^{\beta} P_{p_Q}(x)- D^{\beta} P_\xi(x)\| \\
& \leq C(n,m) \sum_{Q^* \ni x}R (\mathcal{A},\xi,p_Q)  |x-\xi|^{-1}.
\end{align*}
Then, using the multilinearity of $D^{m+1}F(x)$ the bound \ref{lemmaestimates:partbestimateDmplus1F} follows.
\end{proof}

\begin{proof}[Proof of Theorem \ref{thm:main:ext} for $\Gamma = \op{small}$]
Given $\mathcal{A} \in \dot{\op{VJ}}^{m,\omega}_{\mathrm{small}} (E,V),$ we show that $F \in \VC^{m,\omega}_{\mathrm{small}} (\R^n,V)$, i.e. that
$$
\lim_{|x-y| \to 0} \frac{\| D^m F(x)-D^m F (y)\|}{\omega(|x-y|)}=0,
$$
uniformly on $x,y\in \R^n.$ Denote
\begin{equation}\label{eq:definitionS_f}
\op{S}(t):= \sup \left\lbrace   \frac{R (\mathcal{A},u,v)}{\omega(|u-v|)} \, : \,  u,v\in E, \, 0< |u-v| \leq t  \right\rbrace,
\end{equation}
understanding that $S(t)=0$ if there is no such couple $u,v\in E$ with $0< |u-v| \leq t.$ We have that $\op{S}(t)\to 0$ as $t\to 0$ and $\op{S}(t)\leq \|\mathcal{A} \|_{\dot{\op{J}}^{m,\omega}(E,V)}.$ 
We distinguish three possible situations for any couple of distinct points $x,y\in \R^n,$ where at least one of them is outside $E.$ 

\smallskip

\textbf{Case 1.} Assume $x\in \R^n \setminus E$ and $y\in E.$

 By the property \ref{prop:WhitneyPartition}  in Proposition \ref{Wh:estimatedistancecubes}, if $x\in Q^* $ and $u\in E,$ then  $|p_Q-u|  \leq C|x-u|, $ for some absolute constant $C>0.$ Applying this when $u$ is equal to the given $y,$ and also when $u= \xi$ for some $\xi = \xi(x) \in E$ that minimizes $\dist(x,E)=|x-\xi|,$ we obtain that
$$
|p_Q-y | \leq C |x-y|, \quad \text{and} \quad |p_Q-\xi| \leq C |x-\xi| = C \dist(x,E) \leq C |x-y|.
$$
Together with Lemma \ref{lem:derivatives:higherorder}\ref{lemmaestimates:partaapproximationDmF}, and the definition of $\op{S},$ we find that
\begin{align*}
	 \|D^m F(x)-D^m F(y)\| & =	\|D^m F(x)-A_m(y)\|  \\
	 &  \lesssim \sum_{Q^* \ni x} \Big( R (\mathcal{A}, y,p_Q)+  R (\mathcal{A},\xi,p_Q) \Big )  \\
	&  \leq \sum_{Q^*\ni x} \left( \op{S}(\abs{y-p_Q}) \omega(|y-p_Q|) + \op{S}(\abs{\xi-p_Q}) \omega(|\xi-p_Q|)\right)    \\
	&  \lesssim_n  \op{S}(C|x-y|)\omega(|x-y|),
\end{align*}
where we used property \eqref{eq:mod:ext2} of $\omega$ in the last inequality. Therefore
\begin{align}\label{eq:x1}
\frac{	\|D^m F(x)-D^m F(y) \|}{\omega(|x-y|)}  \lesssim  \op{S}(C|x-y|),
\end{align}
and the right-hand side tends to zero as $\abs{x-y}\to 0.$

\smallskip

\textbf{Case 2.} Assume that $x,y\in \R^n \setminus E$ and $\dist([x,y], E) \leq |x-y|.$
Pick $z\in [x,y]$ and $p \in E$ minimizing $\dist([x,y], E) = |z-p|.$ Then,
$$
|x-p| \leq |x-z| + |z-p| \leq |x-y| + |x-y| = 2 |x-y|
$$ 
and the same holds for $|y-p|.$ Then $|x-p|, |y-p| \leq 2|x-y|$ and applying the estimate \eqref{eq:x1} from Case 1 (for $x \in \R^n \setminus E$ and $p\in E,$ and for $y\in \R^n \setminus E$ and $p\in E$) we obtain,
\begin{align*}
	\|D^m F(x)-D^m F(y)\|   &\leq \|D^m F(x)-D^m F(p)\| + \|D^m F(p)-D^m F(y)\|  \\
	&\lesssim  \op{S}(2C|x-y|)\omega(2|x-y|),
\end{align*}
which is a bound of the correct form.

\smallskip

\textbf{Case 3.} Assume $x, y\in  \R^n \setminus E,$ and $\dist([x,y], E) \geq |x-y|.$ Here we use the assumption $\lim_{t \to 0} t/\omega(t)=0.$ 

Let us begin with a computation for arbitrary $z\in \R^n \setminus E.$ Let $\xi = \xi(z) \in E$ minimize $\dist(z,E)=|z-\xi|.$ Then, using first Lemma \ref{lem:derivatives:higherorder}\ref{lemmaestimates:partbestimateDmplus1F}, then the definition of $S$ in \eqref{eq:definitionS_f}, and lastly the point \ref{Wh:estimatedistancecubes} of Proposition \ref{prop:WhitneyPartition}, we get
\begin{align*}
	\|  D^{m+1} F(z)\|    \lesssim \sum_{Q^* \ni z} \frac{R (\mathcal{A},\xi,p_Q)}{|z-\xi|}  \leq \sum_{Q^* \ni z}   \frac{\omega(|p_Q-\xi|) \op{S}(|p_Q-\xi|)}{  \dist(z,E)} \lesssim_n   \frac{\omega(\dist(z,E))\op{S}(C\dist(z,E))}{  \dist(z,E)} .
\end{align*}
In the last inequality, together with \ref{Wh:estimatedistancecubes}, we used the fact that every point of $\R^n \setminus E$ is contained in at most $N(n)$ cubes $Q^*$ as well as property \eqref{eq:mod:ext2} of $\omega.$  
As $D^m F$ is differentiable on $\R^n \setminus E$ and the segment $[x,y]$ lies outside $E,$ by the mean value inequality (for normed-valued mappings) and the above estimate we obtain
\begin{equation}\label{eq:Vzxy}
			\frac{\|D^m F (x)-D^m F(y)\|}{\omega(|x-y|) }  \leq \frac{ |x-y| }{\omega(|x-y|)} \sup_{z\in [x,y]} \| D^{m+1} F(z) \|  \lesssim_n \sup_{z\in [x,y]} U_z(x,y);
\end{equation}
where 
$$
U_z(x,y):=\frac{ |x-y| }{\omega(|x-y|)}\frac{\omega(\dist(z,E))}{ \dist(z,E)} \mathrm{S}(C\dist(z,E)) .
$$ 
Now we fix $z\in [x,y]$ and we consider the behaviour of $\dist(z,E)$ as $|x-y|$ tends to zero. By $\lim_{t \to 0} S(t)=0,$ we have the following: for any $\varepsilon>0,$ let $\delta>0$ be so that $\mathrm{S}(t) \leq \varepsilon,$ whenever $t \leq \delta.$   First, if $C\dist(z,E) \leq \delta,$ then by
 $ |x-y| \leq \dist(z,E)$ and property \eqref{eq:mod:ext1} for $\omega,$ we obtain 
$$
U_z(x,y)\lesssim \mathrm{S}(C\dist(z,E)) \leq \varepsilon.
$$
On the other hand, if $C\dist(z,E) \geq \delta$, we use the property \eqref{eq:mod:ext1} of $\omega$ and that $\mathrm{S} \leq \|\mathcal{A}\|_{\dot{\op{J}}^{m,\omega}(E,V)}$ to bound
$$
U_z(x,y) \lesssim_{\delta}  \frac{ |x-y| }{\omega(|x-y|)}\cdot 1\cdot \|\mathcal{A}\|_{\dot{\op{J}}^{m,\omega}(E,V)} \lesssim \frac{ |x-y| }{\omega(|x-y|)};
$$
now the right-hand side tends to zero as $|x-y|$ tends to $0,$ by the condition \eqref{eq:mod:regularity}. We conclude from the above cases and  \eqref{eq:Vzxy} that $\|D^m F(x)- D^m F(y)\|/\omega(|x-y|) \to 0$ as $|x-y| \to 0,$ uniformly.
\end{proof}

\medskip

\begin{proof}[Proof of Theorem \ref{thm:main:ext} for $\Gamma = \op{large}$]

Let $\mathcal{A} \in \dot{\op{VJ}}^{m,\omega}_{\mathrm{large}} (E,V)$ and we show that $ F \in \VC^{m,\omega}_{\mathrm{large}} (\R^n,V).$
Denote
$$
\op{L}(t):= \sup \left\lbrace \frac{R (\mathcal{A},u,v)}{\omega(|u-v|)} \, : \,  u,v\in E, \,  |u-v| \geq t \right\rbrace,
$$
so that $L(t)\leq \| \mathcal{A}\|_{\dot{\op{J}}^{m,\omega}(E,V)}$, $L$ is non-increasing, and $L(t) \to 0$ as $t\to\infty.$  We understand $S(t)=0$ when there are no $u,v\in E$ with $|u-v| \geq t.$

\smallskip

\textbf{Case 1.} Let $x\in \R^n \setminus E$ and $y\in E.$

For each cube $Q\in \mathcal{Q}$ so that $x\in Q^*,$ we have $|p_Q-y | \leq C |x-y|$ for an absolute constant $C;$ see \ref{Wh:estimatedistancecubes}. Let $\xi = \xi(x) \in E$ minimize $\dist(x,E)=|x-\xi|.$ In this case, although we are letting $|x-y| \to \infty,$ we do not know whether $|p_Q-y|$ or $|p_Q-\xi|$ are large or not. Therefore, we need to study the behaviour of $|p_Q-z|$, for $z\in \lbrace y,\xi \rbrace$ as $|x-y|\to\infty.$  Given $\varepsilon>0, $ we choose $M>0$ large enough so that $ \op{L}(M) \leq \varepsilon.$ Note that such an $M>0$ exists due to the property $\underset{t \to \infty}{\lim} L(t)=0.$ Also, using $\omega(\infty)=\infty,$ we find $K>M$ such that $\omega(t)  \geq \varepsilon^{-1} \omega(M),$ provided that $t\geq K.$

Let us assume that $|x-y| \geq K,$ and examine two cases separately.
If $|p_Q-z| \geq M,$ then $\op{L}(|p_Q-z|) \leq    \op{L}(M)$ and so
\begin{align}\label{eq:x3}
	\frac{ R (\mathcal{A},z,p_Q)}{\omega(|x-y|)} \leq \op{L}(|p_Q-z|) \frac{\omega(|p_Q-z|)}{\omega(|x-y|)}  \leq \op{L}(M) \frac{\omega(C|x-y|)}{\omega(|x-y|)}\lesssim     \op{L}(M) \leq \varepsilon.
\end{align}
On the other hand, if $|p_Q-z|\leq M$, then the assumption $|x-y| >K$ gives
\begin{align}\label{eq:x2}
	 \frac{R (\mathcal{A},z,p_Q)}{\omega(|x-y|)}\leq \| \mathcal{A}\|_{\dot{\op{J}}^{m,\omega}(E,V)}\frac{\omega(|z-p_Q|)}{\omega(|x-y|)}\lesssim \frac{\omega(M)}{\omega(|x-y|)} \leq \varepsilon.
\end{align}

We have used that if $z\in \lbrace y, \xi \rbrace,$ then $|p_Q-z| \leq C |x-z| \leq C |x-y|,$ by virtue of \ref{Wh:estimatedistancecubes}. Thus the bounds \eqref{eq:x2} and \eqref{eq:x3} tell us that regardless of the size of $|p_Q-z|,$ we have 
\begin{align}\label{eq:x4}
\frac{  R (\mathcal{A},z,p_Q)}{\omega(|x-y|)}  \lesssim  \varepsilon, \quad z\in \lbrace y, \xi \rbrace.
\end{align}
From the bound \eqref{eq:x4}, the fact that $D^m F(y)=A_m(y)$, and Lemma \ref{lem:derivatives:higherorder}\ref{lemmaestimates:partaapproximationDmF}, we obtain
\begin{align*}
 	 \frac{\|D^m F (x) - D^m F(y)\| }{\omega(|x-y|)}  = 	\frac{\|D^m F (x) - A_m(y)\| }{\omega(|x-y|)}    \lesssim  \sum_{Q^* \ni x} \left(  \frac{ R (\mathcal{A},y,p_Q)}{\omega(|x-y|)} +  \frac{ R (\mathcal{A},\xi,p_Q)}{\omega(|x-y|)}   \right)   \lesssim_n \varepsilon,
\end{align*}
whenever $|x-y| \geq K.$ 

\smallskip

\textbf{Case 2.} Assume $x,y\in \R^n \setminus E,$ and $\dist([x,y], E) \leq |x-y|.$
Let $p = p(x,y) \in E$ minimize $\dist([x,y], E).$ Then, again, $|x-p|, |y-p| \leq 2 |x-y|.$ As in Case 1, given $\varepsilon>0,$ there exists $M>0$ (independent of $x\in \R^n \setminus E$ and $p\in E$) such that: if $|x-p| \geq M,$ then 
\begin{align}\label{eq:x5}
	\|D^m F(x)- D^m F  (p)\| \leq \varepsilon \omega(|x-p|) \leq   \varepsilon \omega(2 |x-y|) \lesssim \varepsilon \omega(|x-y|).
\end{align}
Now, consider  $|x-p| \leq M$ and let $K>0$ be so large that $\varepsilon\omega(K) >  \omega(M).$ Now if also $|x-y| \geq K,$ then we have 
\begin{align}\label{eq:x6}
	\|D^m F (x)-D^m F(p)\| \leq \| F \|_{\dot C^{m,\omega}(\R^n,V)} \omega(|x-p|) \lesssim \omega(M)\leq \varepsilon\omega(K) \leq \varepsilon \omega(|x-y|).
\end{align}
We used above the fact that the Whitney extension is a bounded operator. Combining \eqref{eq:x5} and \eqref{eq:x6}, we obtain 
\begin{align}\label{eq:x7}
	\sup_{|x-y|>K} \frac{ \|D^m F (x)-D^m F(p) \| }{\omega(|x-y|)}  \lesssim  \varepsilon.
\end{align}
And the same bound holds with $y$ in place of $x;$ thus by triangle inequality,
\begin{align*}
	&\sup_{|x-y|>K} \frac{ \| D^m F(x)-D^m F (y) \| }{\omega(|x-y|)} \\ 
	&\leq \sup_{|x-y|>K} \frac{ \|D^m F (x)-D^m F(p)\| }{\omega(|x-y|)}+\sup_{|x-y|>K} \frac{ \|D^m F (p)-D^m F(y)\| }{\omega(|x-y|)}\lesssim \varepsilon.
\end{align*}

\smallskip

\textbf{Case 3.} Assume $x, y\in  \R^n \setminus E$ 
and $d([x,y],E)>|x-y|.$

For any $\varepsilon>0,$ let $M>0$ be so that $L(M) \leq \varepsilon$ and also choose $K\gg M$ so that $ \varepsilon \omega(K)\geq \omega(M).$ For any two points $u,v\in E,$ we have either $ R(\mathcal{A},u,v)\leq \varepsilon \omega(|u-v|)$ (when $|u-v|\geq M$) or $R (\mathcal{A},u,v)\leq  \|\mathcal{A}\|_{\dot{\op{VJ}}^{m,\omega}(E,V)} \omega(M)$ (when $|u-v|\leq M$). In other words,
\begin{equation}\label{equationgatheringestimates}
R(\mathcal{A},u,v) \lesssim  \max\lbrace \varepsilon \, \omega(|u-v|) , \omega(M)  \rbrace, \quad \text{ for all } \quad u,v\in E.
\end{equation}

Let $z\in [x,y]$ and let $\xi = \xi(z)\in E$ minimize $\dist(z,E) = |z-\xi |.$  Employing first Lemma \ref{lem:derivatives:higherorder}\ref{lemmaestimates:partbestimateDmplus1F}, then \eqref{equationgatheringestimates}, and property \ref{Wh:bdoverlapp} of Proposition \ref{prop:WhitneyPartition}, we derive
\begin{align}
	\|   D^{m+1} F(z) \|   &   \lesssim_n \sum_{Q^* \ni z} \frac{ R(\mathcal{A}, \xi, p_Q) }{\dist(z,E)}  \lesssim_n \max_{Q\in \mathcal{Q} \,  : \, Q^*\ni z} \frac{ R(\mathcal{A}, \xi, p_Q) }{\dist(z,E)}   \label{eq:z00} \\
	& \leq \max_{Q\in \mathcal{Q} \,  : \, Q^*\ni z} \frac{\max\lbrace \varepsilon \, \omega(|p_Q-\xi|) , \omega(M)  \rbrace }{d(z,E)}  \lesssim_n  \frac{\max\lbrace \varepsilon \, \omega(d(z,E)) ,   \omega(M)  \rbrace }{d(z,E)}. \label{eq:z0}
\end{align}
We used the fact that when $z\in Q^*,$ then $|p_Q- \xi | \leq C |z-\xi|.$ Then \eqref{eq:z0} and the mean value inequality (for normed-valued mappings) give 
\begin{equation}\label{eq:z1}
\begin{split}
			\frac{\|D^m F(x)- D^m F(y)\|}{\omega(|x-y|) }  &\leq\frac{ |x-y| }{\omega(|x-y|)} \sup_{z\in [x,y]} \| D^{m+1} F(z) \|  \\ 
	&\lesssim_n \sup_{z\in [x,y]} \frac{ |x-y| }{\omega(|x-y|)} \frac{\max\lbrace \varepsilon \, \omega(d(z,E)) ,   \omega(M)  \rbrace }{d(z,E)} .
\end{split}
\end{equation}
By property \eqref{eq:mod:ext1} of $\omega$ and $\dist(z,E) \geq |x-y|,$ for all $z\in [x,y],$ we continue,
\begin{equation}\label{eq:z2}
\op{RHS}\eqref{eq:z1} \lesssim \sup_{z\in [x,y]}  \max \left\lbrace \varepsilon   ,  \frac{ |x-y| }{\omega(|x-y|)} \frac{\omega(M)}{d(z,E)}  \right\rbrace \leq    \max \left\lbrace \varepsilon   ,  \frac{ \omega(M)}{\omega(|x-y|)}   \right\rbrace.
\end{equation}
Recalling that $\varepsilon \omega(K)\geq \omega(M)$ we conclude
$$
\sup_{|x-y|>K}\frac{\|D^m F(x)- D^m F(y)\|}{\omega(|x-y|) } \lesssim   \sup_{|x-y|>K}  \max \left\lbrace \varepsilon   ,  \frac{ \omega(M)}{\omega(|x-y|)}   \right\rbrace \leq \varepsilon.
$$
\end{proof}

\begin{proof}[Proof of Theorem \ref{thm:main:ext} when $\Gamma =\op{far}$]
Here we assume that $\omega(\infty)= \infty$.
We denote 
\begin{align*}
	\op{D}(t):= \sup \left\lbrace \frac{R(\mathcal{A},u,v)}{\omega(|u-v|)} \, : \,  u,v\in E, \,  \max(|u|, |v|) \geq t \right\rbrace,
\end{align*}
so that $\op{D}(t)\leq \| \mathcal{A} \|_{\dot{\op{J}}^{m,\omega}(E,V)}$ and $\mathrm{D}$ is non-increasing, and moreover, thanks to Lemma \ref{lem:dist=far}, that $\op{D}(t)\to 0$ as $t\to \infty$. We show that
$$
\lim_{K \to \infty} \sup_{\substack{x,y\in \R^n \\ \min(|x|, |y|) \geq K}} \frac{\|D^m F(x)-D^m F(y)\|}{\omega(|x-y|)} =0.
$$

\smallskip

\textbf{Case 1}. Let $x\in \R^n \setminus E,$ $y\in E$. Given $\varepsilon>0,$ let $M>0$ be such that $\mathrm{D}(M) \leq \varepsilon$. We consider $|x| \geq K\gg M$ for a large constant $K$ to be soon determined. Let $\xi = \xi(x) \in E$ minimize $\dist(x,E)=|x-\xi|.$  Applying Lemma \ref{lem:derivatives:higherorder}\ref{lemmaestimates:partaapproximationDmF} and bearing in mind the property \ref{Wh:bdoverlapp}, we find a cube $Q_x\in \mathcal{Q}$ with $x\in (Q_x)^*$, and $z\in \lbrace \xi, y \rbrace$ so that 
\begin{align}\label{eq:v2}
	\frac{\|D^m F(x)-D^m F(y)\|}{\omega(|x-y|)}  & \leq \frac{\sum_{Q^* \ni x}R (\mathcal{A}, y,p_Q)+  R (\mathcal{A},\xi,p_Q) }{\omega(|x-y|)} \lesssim_n \frac{R (\mathcal{A}, z,p_{Q_x})     }{\omega(|x-y|)} \nonumber  \\
	& \leq \min\lbrace \mathrm{D}(|p_{Q_x}|), \mathrm{D}(|z|), \|\mathcal{A}\|_{\dot{\op{J}}^{m,\omega}(E,V)} \rbrace \frac{\omega(|p_{Q_x}-z|)}{\omega(|x-y|)} .
\end{align}

The last inequality is a consequence of the definition of $\mathrm{D}(t).$ Now, if either $|p_{Q_x}| \geq M$ or $|z|\geq M,$ the minimum in the term \eqref{eq:v2} is smaller than $\varepsilon$ due to the choice of $M.$ And because $|p_{Q_x}-z|\leq C|x-y|$ for an absolute constant $C>0$ (see \ref{Wh:estimatedistancecubes}); we conclude that \eqref{eq:v2} is bounded by an absolute multiple of $\varepsilon$ in this particular case. And if $|p_{Q_x}| \leq M$ and $|z|\leq M,$ then $|p_{Q_x}-z| \leq 2M,$ and so $\omega(|p_{Q_x}-z|) \lesssim \omega(M)$ and we bound
$$
\op{RHS}\eqref{eq:v2}\lesssim  \frac{\omega(M)}{\omega( ||x|-M|)} \lesssim  \frac{\omega(M)}{\omega\left( K- M \right)} .
$$
By $\omega(\infty) =\infty$ the right-hand side is smaller than $\varepsilon$ provided that $K\gg M$ is taken sufficiently large.

\smallskip

\textbf{Case 2}. Assume $x,y \in \R^n \setminus E$ and $d([x,y],E) \leq |x-y|.$

Let $\varepsilon>0$ and let $R>0$ be as in Case 1. Suppose that $|x|, |y| \geq K.$ Let $p \in E$ be such that $d([x,y], E) =d([x,y], p),$ and then $|x-p|, |y-p| \leq 2|x-y|,$ as we have already seen several times.  By Case 1 applied to the pairs of points $x\in \R^n \setminus E,$ $p\in E$ and $y \in \R^n \setminus E$, $p\in E$, 
we have
\[
\|D^m F(x)-D^m F(p)\| \lesssim \varepsilon \omega(|x-p|),\qquad \| D^m F(y)- D^m F(p)\| \lesssim \varepsilon \omega(|y-p|). 
\]
Thus the claim follows by triangle inequality.

\textbf{Case 3}. Assume $x,y \in \R^n \setminus E$ and $d([x,y],E) \geq |x-y|.$

Let $\varepsilon>0,$ and $K \gg M$ be the parameters we used in Case 1. Note that we can enlarge $K$ if necessary so as to satisfy 
\begin{equation}\label{eq:farCase3:choiceK}
K \geq 2M \quad \text{and} \quad  \omega(M) \leq \varepsilon \omega(K).
\end{equation}
Note that condition \eqref{eq:mod:coer} guarantees the existence of such $K.$ Let us also assume that $|x| \geq K.$  Following the lines \eqref{eq:z00} and \eqref{eq:z1}, we find $z\in [x,y]$, a point $\xi = \xi(z) \in E$ minimizing $\dist(z,E) = |z-\xi |$, and a cube $Q_z \in \mathcal{Q}$ with $z\in Q^*$ so that
\begin{equation}\label{eq:farCase3:inicial}
			\frac{\|D^m F(x)-D^m F(y)\|}{\omega(|x-y|)} \lesssim_n \frac{|x-y|}{\omega(|x-y|)} \frac{R(\mathcal{A},\xi,p_{Q_z})}{\dist(z,E)}   .
\end{equation}
First of all, observe that when $| \xi | \geq M$ or $|p_{Q_z}| \geq M$, then $R(\mathcal{A},\xi,p_{Q_z}) \leq \varepsilon \omega(|\xi-p_{Q_z}|),$ with $|\xi-p_{Q_z}| \leq C |\xi-z| =C d(z,E).$ Thus, in this particular subcase, using \eqref{eq:mod:ext1} we can estimate
\begin{equation}\label{eq:farCase3:easycase}
\op{RHS}\eqref{eq:farCase3:inicial} \lesssim \frac{|x-y|}{\omega(|x-y|)} \frac{\omega(\dist(z,E))}{\dist(z,E)} \varepsilon \lesssim_{C_\omega} \varepsilon.
\end{equation}
 
Therefore, we will assume from now on that $|\xi|, |p_{Q_z}| \leq M.$ Now, we look at the last numerator in \eqref{eq:farCase3:inicial}; observe that, by the definition of $R(\mathcal{A},\xi,p_{Q_z}),$ and the fact that $F$ extends the jet $\mathcal{A}$ from $E$ to $\R^n,$ we can find some $k \in \lbrace 0, \ldots ,m \rbrace$ and some $\widetilde{\xi} \in [\xi, p_{Q_z}]$ so that
\begin{align}\label{eq:farCase3:taylorestimate}
 R(\mathcal{A},\xi,p_{Q_z}) & \leq   \frac{\|A_k(\xi)-\sum_{j=0}^{m-k} \frac{1}{j!} A_{k+j}(p_{Q_z})(\xi-p_{Q_z}) \|}{ |\xi-p_{Q_z}|^{m-k}  } \nonumber \\
 &  =\frac{\|D^k F(\xi)-\sum_{j=0}^{m-k} \frac{1}{j!} D^{k+j} F(p_{Q_z})(\xi-p_{Q_z}) \| }{|\xi-p_{Q_z}|^{m-k}} \lesssim   \|D^m F(\widetilde{\xi})-D^m F(p_{Q_z})\|.
\end{align}
Now, since $F\in \dot{C}^{m,\omega}(\R^n,V)$ with $\|F\|_{\dot{C}^{m,\omega}(\R^n,V)} \lesssim_{n,m,C_\omega} \|\mathcal{A}\|_{\dot{\op{J}}^{m,\omega}(E,V)},$ and $| \widetilde{\xi}-p_{Q_z}| \leq | \xi-p_{Q_z}| \leq 2 M,$ we can combine the estimates \eqref{eq:farCase3:inicial} and \eqref{eq:farCase3:taylorestimate} to derive
\begin{align}\label{eq:farCase3final}
\frac{\|D^m F(x)-D^m F(y)\|}{\omega(|x-y|)} & \lesssim_{n,m,C_\omega} \frac{|x-y|}{\omega(|x-y|)} \frac{\|D^m F(\widetilde{\xi})-D^m F(p_{Q_z})\|}{d(z,E)} \nonumber \\
&  \lesssim_{ \|\mathcal{A}\|_{\dot{\op{J}}^{m,\omega}(E,V)}} \frac{|x-y|}{\omega(|x-y|)} \frac{\omega(|\widetilde{\xi}-p_{Q_z}|)}{d(z,E)} \lesssim_{C_\omega} \frac{|x-y|}{\omega(|x-y|)} \frac{\omega(M)}{d(z,E)}.
\end{align}

To complete the proof, observe that \eqref{eq:farCase3:choiceK} gives
$$
d(z,E) = |z- \xi| \geq |z|-|\xi| \geq K - M \geq K/2,
$$
and $\omega(M) \leq \varepsilon \omega(K) \lesssim_{C_\omega} \varepsilon \omega(d(z,E)).$ By plugging this estimate into $\op{RHS}$\eqref{eq:farCase3final} we obtain
$$
\frac{\|D^m F(x)-D^m F(y)\|}{\omega(|x-y|)} \lesssim \frac{|x-y|}{\omega(|x-y|)} \frac{\omega(d(z,E))}{d(z,E)} \varepsilon \lesssim_{C_\omega} \varepsilon;
$$
the last bound being a consequence of \eqref{eq:mod:ext1} and $|x-y| \leq d(z,E).$

\end{proof}

\bibliographystyle{abbrv}
\bibliography{references.bib}

\end{document}